\documentclass[a4paper,12pt,leqno]{amsart}
\usepackage{geometry}
\usepackage[utf8]{inputenc}
\usepackage{amsmath}
\usepackage{dsfont, mathtools}
\usepackage{amssymb}
\usepackage[natbib=true, style=numeric, sorting=nyvt]{biblatex}
\usepackage[colorlinks=true,allcolors=blue, urlcolor=black]{hyperref}
\newtheorem{theorem}{Theorem}
\newtheorem{remark}{Remark}[subsection]

\newtheorem{proposition}{Proposition}
\newtheorem{lemma}{Lemma}
\newtheorem{ha}{Harmonic Analysis Result}
\newtheorem{ep}{Euler Product Result}
\newtheorem{nt}{Number Theory Result}
\newtheorem{probres}{Probability Result}
\newcommand{\N}{\mathbb{N}}
\newcommand{\R}{\mathbb{R}}
\newcommand{\E}{\mathbb{E}}
\newcommand{\p}{\mathbb{P}}

\newcommand{\var}{\mathrm{Var}}
\author{Seth Hardy}
\title[Almost sure bounds for a Steinhaus random multiplicative function]{Almost sure bounds For a weighted Steinhaus random multiplicative function}
\address{Mathematics Institute, Zeeman Building, University of Warwick, Coventry CV4 7AL, England}
\email{Seth.Hardy@warwick.ac.uk}
\date{\today}
\thanks{The author is supported by the Swinnerton-Dyer scholarship at the Warwick Mathematics Institute Centre for Doctoral Training.}

\begin{document}
\maketitle

\begin{abstract}
We obtain almost sure bounds for the weighted sum $\sum_{n \leq t} \frac{f(n)}{\sqrt{n}}$, where $f(n)$ is a Steinhaus random multiplicative function. Specifically, we obtain the bounds predicted by exponentiating the law of the iterated logarithm, giving sharp upper and lower bounds.
\end{abstract}

\section{Introduction}
The \emph{Steinhaus random variable} is a complex random variable that is uniformly distributed on the unit circle $\{ z : \, |z| = 1 \}$ in the complex plane. Letting $(f(p))_{p \text{ prime}}$ be independent Steinhaus random variables, we define the \emph{Steinhaus random multiplicative function} to be the (completely) multiplicative extension of $f$ to the natural numbers. That is \[ f(n) = \prod_{p \mid n} f(p)^{v_p(n)}, \] where $v_p (n)$ is the $p \,$-adic valuation of $n$. Weighted sums of Steinhaus $f(n)$ were studied in recent work of \citet{AHZ} as a model for the Riemann zeta function on the critical line. Noting that \[ \zeta(1/2 + it) = \sum_{n \leq |t|} \frac{1}{n^{1/2 + it}} + o(1), \] they modelled the zeta function at height $t$ on the critical line by the function
 \[ M_f (t) \coloneqq \sum_{n \leq t} \frac{f(n)}{\sqrt{n}}, \]
 for $f$ a Steinhaus random multiplicative function. The motivation for this model is that the function $n^{-it}$ is multiplicative, it takes values on the complex unit circle, and $(p^{-it})_{p \text{ prime}}$ are asymptotically independent for any finite collection of primes. \\

In their work studying $M_f (t)$, Aymone, Heap, and Zhao proved an upper bound analogous to a conjecture of \citet{FGH} on the size of the zeta function on the critical line, which states that 
\[ \max_{t \in [T,2T]} |\zeta(1/2 + it)| = \exp\Biggl( (1 + o(1)) \sqrt{\frac{1}{2} \log T \log \log T} \Biggr) .\]
Due to the oscillations of the zeta function, the events that model this maximum size involve sampling $T \log T$ independent copies of $M_f (t)$. \\

Despite being the ``wrong'' object to study with regards to the maximum of the zeta function, one may also wish to find the correct size for the \emph{almost sure} large fluctuations of $M_f(x)$, since this is an interesting problem in the theory of random multiplicative functions. In this direction, Aymone, Heap, and Zhao obtained an upper bound of
\[ M_f (x) \ll (\log x)^{1/2 + \varepsilon}, \]
almost surely, for any $\varepsilon > 0$. This is on the level of squareroot cancellation, since $M_f(x)$ has variance of approximately $\log x$. Furthermore, they obtained the lower bound that for any $L>0$,
\[ \limsup_{x \rightarrow \infty} \frac{|M_f (x)|}{\exp\bigl((L+o(1))\sqrt{\log \log x}\bigr)} \geq 1, \]
almost surely. If close to optimal, this lower bound demonstrates a far greater degree of cancellation than the upper bound, and suggests that $M_f$ is being dictated by its Euler product. One may expect that
\[ |M_f (x)| \approx \biggl|\prod_{p \leq x} \bigl( 1 - f(p)/\sqrt{p} \bigr)^{-1} \biggr| \approx \exp \Biggl(\sum_{p \leq x} \frac{\Re f(p) }{ \sqrt{p}} \Biggr),  \]
and the law of the iterated logarithm (see, for example, \citet{Gut_2013}, chapter 8) suggests that
\[ \limsup_{x \rightarrow \infty} \frac{\sum_{p \leq x} \Re ( f(p) ) / \sqrt{p}}{\sqrt{\log_2 x \log_4 x}} = 1 \, ,\]
where $\log_k$ denotes the $k$-fold iterated logarithm. In this paper we prove the following results, which confirm the strong relation between $M_f (x)$ and the Euler product of $f$.
\begin{theorem}[Upper Bound] \label{T:1}
For any $\varepsilon > 0$, we have 
\[M_f (x) \ll \exp{\bigl((1+\varepsilon)\sqrt{ \log_2 x \log_4 x}\bigr)} \, , \]
almost surely.
\end{theorem} 
\begin{theorem}[Lower Bound] \label{T:2}
For any $\varepsilon > 0$, we have 
\[ \limsup_{x\rightarrow \infty} \frac{|M_f(x)|}{ \exp{\bigl((1-\varepsilon)\sqrt{ \log_2 x \log_4 x}\bigr)}} \geq 1 \, ,\]
almost surely.
\end{theorem}
These are the best possible results one could hope for, with upper and lower bounds of the same shape, matching the law of the iterated logarithm. \\

One of the most celebrated upper bound results in the literature is that of \citet{LTW}, who found an upper bound for unweighted partial sums of the Rademacher multiplicative function. Originally introduced by \citet{wintner1944random} as a model for the Möbius function, the Rademacher random multiplicative function is the multiplicative function supported on square-free integers, with $(f(p))_{p \text{ prime}}$ independent and taking values $\{ -1, 1 \}$ with probability $1/2$ each. In this paper, Wintner showed that for Rademacher $f$ we have roughly squareroot cancellation, in that
\[ \sum_{n \leq x} f(n) \ll x^{1/2 + \varepsilon}, \]
almost surely, for any $\varepsilon>0$. Lau, Tenenbaum, and Wu obtained a far more precise result, proving that for Rademacher $f$,
\[ \sum_{n \leq x} f(n) \ll \sqrt{x}(\log \log x)^{2 + \varepsilon},\]
almost surely, for any $\varepsilon>0$, and recent work of \citet{caich2023sure} has improved this result. Indeed, we find that similar techniques to those of Lau, Tenenbaum, and Wu, as well as more recent work on connecting random multiplicative functions to their Euler products (see \citet{harper2020moments}) lead to improvements over the bounds from \citet{AHZ}. Note that the weights $\frac{1}{\sqrt{n}}$ in the sum $M_f (x)$ give a far stronger relation to the underlying Euler product of $f$ than in the unweighted case, so finding the ``true size'' of large fluctuations is relatively more straightforward.
\subsection{Outline of the proof of Theorem \ref{T:1}}
For the proof of the upper bound we first partition the natural numbers into intervals, say $[x_{i-1},x_i)$, so that $M_f (x)$ doesn't vary too much over these intervals. If the fluctuations of $M_f (x)$ between test points $(x_i)$ is small enough, then it suffices to just get an upper bound only on these $(x_i)$. This is the approach taken by both \citet{AHZ} and \citet{LTW}. The latter took this a step further and considered each test point $x_i$ as lying inside some larger interval, say $[X_{l-1}, X_l)$. These larger intervals determine the initial splitting of our sum, which takes the shape 
\[ M_f (x_i) = \sum_{\substack{n \leq x_i \\ P(n) \leq y_0}} \frac{f(n)}{\sqrt{n}} + \sum_{\substack{y_{j-1} < m \leq x_i \\ p | m \Rightarrow p \in (y_{j-1}, y_j]}} \frac{f(m)}{\sqrt{m}} \sum_{\substack{n \leq x_i / m \\ P(n) \leq y_{j-1}}} \frac{f(n)}{\sqrt{n}}, \]
with the parameters $(y_j)_{j=0}^J$ depending on $l$. One finds that the first term and the innermost sum of the second term behave roughly like $F_{y_j} (1/2)$, for $y_j$ the smoothness parameter, where $F_y (s) \coloneqq \prod_{p \leq y} ( 1 - f(p)/p^s)^{-1}$. Obtaining this relation is a critical step in our proof. The first sum can be seen to behave like the Euler product $F_{y_0} (1/2)$ by simply completing the range $n \leq x_i$ to all $n \in \N$. The inner sum of the second term is trickier, and we first have to condition on $f(p)$ for $y_{j-1} < p \leq y_j$ in the outer range so that we can focus entirely on understanding these inner sums over smooth numbers. Having conditioned, it is possible for us to replace our outer sums with integrals, allowing application of the following key result, which has seen abundant use in the study of random multiplicative functions (see for example \citet{gerspach2022low}, \citet{HarperHighMom}, \cite{harper2020moments}, or \citet{Mastrostefano}).
\begin{ha}[(5.26) of \citet{montgomery2007multiplicative}] \label{ha-result}
Let $(a_n)_{n=1}^\infty$ be a sequence of complex numbers, and let $A(s) = \sum_{n=1}^\infty \frac{a_n}{n^s}$ denote the corresponding Dirichlet series, and $\sigma_c$ the abscissa of convergence. Then for any $\sigma > \max\{ 0, \sigma_c \}$, we have \[ \int_0^\infty \frac{|\sum_{n \leq x} a_n|^2}{x^{1 + 2 \sigma}} \, dx \, = \frac{1}{2\pi} \int_{-\infty}^{\infty} \biggl| \frac{A(\sigma + i t)}{\sigma + it} \biggr|^2 \, dt \, . \]
\end{ha} 
It is then a case of extracting the Euler product from the integral. To do this, we employ techniques from \citet{gerspach2022low}, noting that some factors of the Euler product remain approximately constant over small ranges of integration. We then show that these Euler products don't exceed the anticipated size coming from the law of the iterated logarithm. To do this, we consider a sparser third set of points, $(\tilde{X}_{k})$, chosen so that the variance of $\sum_{p \leq \tilde{X}_k}\Re f(p)/\sqrt{p}$ grows geometrically in $k$. These intervals mimic those used in classical proofs of the law of the iterated logarithm (for example, in chapter 8 of \citet{Gut_2013}), and are necessary to obtain a sharp upper bound by an application of Borel--Cantelli. \\

\subsection{Outline of the proof of Theorem \ref{T:2}}
The proof of the lower bound is easier, instead relying on an application of the second Borel--Cantelli lemma. The aim is to show that, for some appropriately chosen points $(T_k)$, the function $|M_f (t)|$ takes a large value between $T_{k-1}$ and $T_k$ infinitely often with probability $1$. We begin by noting that
\[ \max_{t \in [T_{k-1}, T_k]} | M_f (t) |^2 \geq \frac{1}{\log T_k} \int_{T_{k-1}}^{T_k} \frac{\bigl|M_f (t)\bigr|^2}{t^{1 + \sigma}} \, dt \, ,\]
for some small convenient $\sigma > 0$. Over this interval we have $M_f (t) = \sum_{n \leq t \, : \, P(n) \leq T_k} f(n)/\sqrt{n}$, and so we may work with this instead. We now just need to complete the integral to the range $[1,\infty)$ so that we can apply Harmonic Analysis Result \ref{ha-result}, and again obtain the Euler product. This can be done by utilising the upper bound from Theorem \ref{T:1} to complete the lower range of the integral, and an application of Markov's inequality shows that the contribution from the upper range is almost surely small when $\sigma$ is chosen appropriately. After some standard manipulations to remove the integral on the Euler product side, one can find that, roughly speaking,
\[ \max_{t \in [T_{k-1}, T_k]} |M_f (t)|^2 \geq \exp\biggl( 2 \sum_{p \leq T_k} \frac{\Re f(p)}{\sqrt{p}} \biggr) + O(E(k)), \]
occurs infinitely often almost surely, for some relatively small error term $E(k)$.
The proof is then completed using the Berry-Esseen Theorem and the second Borel--Cantelli lemma, following closely a standard proof of the law of the iterated logarithm (this time we follow Varadhan, \cite{varadhan2001probability}, section 3.9).
\section{Upper bound} \label{s:mainarg}
\subsection{Bounding variation between test points} \label{s:testpointvariation}
We first introduce a useful lemma that will be used for expectation calculations throughout the paper.
\label{s:max}
\begin{lemma} \label{L:Expectation}
Let $\{a(n)\}_{n \in \N}$ be a sequence of complex numbers, with only finitely many $a(n)$ nonzero. For any $l \in \N$, we have
\[ \E \biggl| \sum_{n \geq 1} \frac{a(n) f(n)}{\sqrt{n}} \bigg|^{2l} \leq \biggl( \sum_{n \geq 1} \frac{|a(n)|^2 \tau_l (n)}{n} \biggr)^l, \]
where $\tau_l$ denotes the $l$-divisor function, $\tau_l (n) = \# \{(a_1,...,a_l): a_1 a_2 ... a_l = n  \}$, and we write $\tau(n)$ for $ \tau_2 (n)$.
\end{lemma}

\begin{proof}
This is Lemma 9 of \cite{AHZ}. It is proved by conjugating, taking the expectation, and applying Cauchy--Schwarz.
\end{proof}

\begin{lemma} \label{L:Max}
There exists a small constant $c \in (0,1)$, such that, with \[ x_i = \lfloor e^{i^c} \rfloor, \]
we have the bound
\[
\max_{x_{i-1} < x \leq x_i} | M_f (x) - M_f (x_{i-1}) | \ll 1 \; \text{a.s.}
\]
\end{lemma}
\begin{proof}
This result closely resembles Lemma 2.3 of \citet{LTW}, who proved a similar result for (unweighted) Rademacher $f$. We note that their lemma purely relies on the fourth moment of partial sums of $f(n)$ being small. For $f$ Steinhaus, an application of Lemma \ref{L:Expectation} implies that for $ u \leq v$,
\begin{align*}
\E \bigl| \sum_{u < n \leq v} f(n) \bigr|^4 \leq \Bigl(\sum_{u < n \leq v} \tau(n) \Bigr)^2.
\end{align*}
Now, if additionally $u \asymp v$, then by Theorem 12.4 of Titchmarsh \cite{titchmarsh1986theory}, we have
\begin{align*} 
\sum_{u < n \leq v} \tau(n) & = v \log v - u \log u + (2 \gamma - 1)(v-u) + O(v^{1/3}) \\ 
&= (v-u) \log u + v \log (v/u) + (2 \gamma - 1)(v-u) + O(v^{1/3}) \\
&\ll (v-u) \log u + O(v^{1/3}) .
\end{align*}
So certainly
\[ \E \bigl| \sum_{u < n \leq v} f(n) \bigr|^4 \ll v^{2/3} (v-u)^{4/3} (\log v)^{52/3}, \]
which is the fourth moment bound in the work of \citet{LTW} (equation (2.5)). Note that it suffices to consider $u \asymp v$, since for $c \in (0,1)$, we have $x_{i-1} \asymp x_i$. The rest of their proof then goes through for Steinhaus $f$, so that for some $c \in (0,1)$, we have
\begin{align*}
\max_{x_{i-1} < x \leq x_i} \bigl| \sum_{x_{i-1} < n \leq x} f(n) \bigr| \ll \frac{\sqrt{x_i}}{\log x_i}. 
\end{align*}
It then follows from Abel summation that
\[ \max_{x_{i-1} < x \leq x_i} \bigl| \sum_{x_{i-1} < n \leq x} \frac{f(n)}{\sqrt{n}} \bigr| \ll \sqrt{\frac{x_i}{x_{i-1}}} \frac{1}{\log x_i} \ll 1,  \]
as required. We fix the value of $c \in (0,1)$ for the remainder of this section, and remark that this bound is stronger than we need.
\end{proof}
\subsection{Bounding on test points} \label{ss:boundingontestpoints}
To complete the proof of Theorem \ref{T:1}, it suffices to prove the following proposition \begin{proposition} \label{P:Mxibound} For any $\varepsilon>0$, we have
\[ M_f (x_i) \ll \exp{\bigl((1+\varepsilon)\,\sqrt{\log_2 x_i \log_4 x_i}\bigr)}, \hspace{0.3cm} \forall \, i , \] almost surely.
\end{proposition} 
\begin{proof}[Proof of Theorem \ref{T:1}, assuming Proposition \ref{P:Mxibound}]
By the triangle inequality, we have
\[|M_f(x)| \leq |M_f(x_{i-1})| + \max_{x_{i-1} < x \leq x_i} | M_f (x) - M_f (x_{i-1}) |.\]
Theorem \ref{T:1} then follows from Proposition \ref{P:Mxibound} (which bounds the first term) and Lemma \ref{L:Max} (which bounds the second term).
\end{proof}
The rest of this section is devoted to proving Proposition \ref{P:Mxibound}. We begin by fixing  $\varepsilon>0$. Throughout we will assume this is sufficiently small, and implied constants (from $\ll$ or ``Big Oh'' notation) will depend only on $\varepsilon$, unless stated otherwise. Beginning similarly to \citet{LTW}, we define the points $X_l = e^{e^l}$, and for some $\alpha \in (0,1/2)$ chosen at the end of subsection \ref{s:Dijterm}, we define 
\begin{align}
y_0 = \exp\biggl(\frac{ce^l}{6l}\biggr), \;\; y_j = y_{j-1}^{e^{\alpha}}, \text{ for $1 \leq j \leq J$ },
\label{equ:2.01} \tag{2.01}
\end{align}
where $J$ is minimal so that $y_J \geq X_l$. One can calculate that
\begin{equation}
J \ll \frac{\log l}{\alpha}.
\label{equ:2.02} \tag{2.02}
\end{equation}
The points $X_l$ partition the positive numbers so that each $x_i$ lies inside some interval $[X_{l-1}, X_l)$. As mentioned, we also consider $X_{l-1}$ as being inside some very large intervals $[\tilde{X}_{k-1}, \tilde{X}_k)$, where $\tilde{X}_k = \exp(\exp(\rho^k))$ for some $\rho>1$ depending only on $\varepsilon$, specified at the end of subsection \ref{s:loiltyperesult}. Throughout we will assume that $k$, and subsequently $i$ and $l$, are sufficiently large. To prove Proposition \ref{P:Mxibound}, it suffices to show that the probability of
\begin{equation*} \mathcal{A}_k = \Biggl\{ \sup_{\tilde{X}_{k-1} \leq X_{l-1} < \tilde{X}_k} \sup_{X_{l-1} \leq x_i < X_l} \frac{|M_f (x_i)|}{\exp\bigl((1 + \varepsilon){\sqrt{\log_2 x_i \log_4 x_i}}\bigl)} > 4 \Biggr\}, \end{equation*}
is summable in $k$, since this will allow for application of the first Borel--Cantelli lemma. As mentioned, we first split the sum according to the prime factorisation of each $n$,
\[ M_f (x_i) = S_{i,0} + \sum_{1 \leq j \leq J} S_{i,j}, \]
where
\begin{align*}
S_{i,0} &= \sum_{\substack{n \leq x_i \\ P(n) \leq y_0}} \frac{f(n)}{\sqrt{n}}, \label{equ:2.03} \tag{2.03} \\
S_{i,j} &= \sum_{\substack{y_{j-1} < m \leq x_i \\ p | m \Rightarrow p \in (y_{j-1}, y_j]}} \frac{f(m)}{\sqrt{m}} \sum_{\substack{n \leq x_i / m \\ P(n) \leq y_{j-1}}} \frac{f(n)}{\sqrt{n}}. \label{equ:2.04} \tag{2.04}
\end{align*}
It is fairly straightforward to write $S_{i,0}$ in terms of an Euler product by completing the sum over $n$. The $S_{i,j}$ terms are a bit more complicated, and we will have to do some conditioning to obtain the Euler products which we expect dictate the inner sums. Similar ideas play a key role in the work of \citet{gerspach2022low}. With this in mind, we have
\begin{equation} \p (\mathcal{A}_k) \leq \p(\mathcal{B}_{0,k}) + \p(\mathcal{B}_{1,k}), \label{equ:2.05} \tag{2.05} \end{equation}
where
\begin{align*}
\mathcal{B}_{0,k} &= \Biggl\{\sup_{\tilde{X}_{k-1} \leq X_{l-1} < \tilde{X}_k} \sup_{X_{l-1} \leq x_i < X_l} \frac{|S_{i,0}|}{\exp\bigl((1 + \varepsilon){\sqrt{\log_2 x_i \log_4 x_i}}\bigr)} > 2 \Biggr\}, \label{equ:2.06} \tag{2.06}\\
\mathcal{B}_{1,k} &= \Biggl\{\sup_{\tilde{X}_{k-1} \leq X_{l-1} < \tilde{X}_k} \sup_{X_{l-1} \leq x_i < X_l} \frac{ \sum_{1 \leq j \leq J} |S_{i,j}|}{\exp\bigl((1 + \varepsilon){\sqrt{\log_2 x_i \log_4 x_i}}\bigr) } > 2 \Biggr\} .
\end{align*}
It suffices to prove that both $ \p(\mathcal{B}_{0,k})$ and $\p(\mathcal{B}_{1,k})$ are summable. 

\subsection{Conditioning on likely events} 
To proceed, we will utilise the following events, recalling that $F_y (s) = \prod_{p \leq y} ( 1 - f(p)/p^s)^{-1}$.
\begin{align*}
G_{j,l} &= \Biggl\{\frac{ \sup_{p \leq y_{j-1}} |F_p (1/2)|}{\exp\bigl((1 + \varepsilon)\sqrt{ \log_2 {X}_{l-1} \log_4 {X}_{l-1}}\bigr)} \leq \frac{1}{l^5}  \Biggr\}, \label{equ:2.07} \tag{2.07} \\
I_{j,l}^{(1)} &= \Biggl\{ \int_{-1/\log y_{j-1}}^{1/\log y_{j-1}}  \Bigg|\frac{F_{y_{j-1}}(1/2 + 1/\log X_l + it)}{F_{y_{j-1}} (1/2)} \Bigg|^2 \, dt \leq \frac{l^4}{\log y_{j-1}}\Biggr\} , \\
I_{j,l}^{(2)} &= \Biggl\{ \sum_{\substack{1/\log y_{j-1} \leq |T| \leq 1/2 \\ T \text{ dyadic}}} \frac{1}{T^2} \int_{T}^{2T} \Bigg| \frac{F_{y_{j-1}}(1/2 + 1/\log X_l + it)}{F_{e^{1/T}} (1/2)} \Bigg|^2  \, dt \, \leq l^4 \log y_{j-1} \Biggr\} , \\
I_{j,l}^{(3)} &= \Biggl\{ \int_{1/2}^{\infty} \frac{|F_{y_{j-1}}(1/2 + 1/\log X_l + it)|^2 + |F_{y_{j-1}}(1/2 + 1/\log X_l - it)|^2}{t^2} \, dt \, \leq l^4 \log y_{j-1} \Biggr\}. 
\end{align*} 
\begin{remark} \label{r:integralsignremark}
\normalfont The summand in the events $I_{j,l}^{(2)}$ should be adjusted for negative $T$, in which case one should flip the range of integration, and instead take $F_{e^{1/|T|}}(1/2)$ in the denominator of the integrand. For the sake of tidiness, we have left out these conditions.
\end{remark}
These events will be very useful to condition on when it comes to estimating the probabilities in \eqref{equ:2.06}. Ideally, \emph{all} of these events will occur eventually, and we will show that this is the case with probability one. Therefore, we define the following intersections of these events, giving ``nice behaviour'' for $S_{i,j}$ for all $i,j$ where $x_i$ runs over the range $[X_{l-1},X_l)$ for $X_{l-1} \in [\tilde{X}_{k-1},\tilde{X}_k)$. We stress that $J$ (defined in \eqref{equ:2.01}) depends on $l$ .
\begin{align*}
G_k = \bigcap_{l \,:\, \tilde{X}_{k-1} \leq X_{l-1} < \tilde{X}_k} \bigcap_{j=1}^J G_{j,l} \, ,  &&  I_{j,l} = \bigcap_{r=1}^3 I_{j,l}^{(r)}  \, ,  &&  I_{k} = \bigcap_{l \, : \, \tilde{X}_{k-1} \leq X_{l-1} < \tilde{X}_k}\bigcap_{j=1}^J I_{j,l} \, . \label{equ:2.08} \tag{2.08}
\end{align*}
\begin{proposition} \label{p:sumthenprop}
Proposition \ref{P:Mxibound} follows if $\p(G_k^c)$ and $ \p(I_{k}^c)$ are summable.
\end{proposition}

We will later show that $\p(G_k^c)$ and $\p(I_{k}^c)$ are indeed summable in subsections \ref{s:loiltyperesult} and \ref{s:cond} respectively. We proceed with proving this proposition, which is quite difficult and constitutes a large part of the paper.

\begin{proof}[Proof of Proposition \ref{p:sumthenprop}]

First we will show that $\p(\mathcal{B}_{0,k})$ is summable. It follows from definition \eqref{equ:2.03} that
\[ S_{i,0} = F_{y_0} (1/2) - \sum_{\substack{n > x_i \\ P(n) \leq y_0}} \frac{f(n)}{\sqrt{n}}. \]
By the triangle inequality (recalling \eqref{equ:2.06}), we have
    \begin{align*} \p( \mathcal{B}_{0,k}) & \leq \p \Biggl( \sup_{\tilde{X}_{k-1} \leq X_{l-1} < \tilde{X}_k} \frac{\big|F_{y_{0}} \bigl(1/2 \bigr)\big|}{\exp\Bigl((1 + \varepsilon)\sqrt{\log_2 {X}_{l-1} \log_4 {X}_{l-1}}\Bigr)} >  1  \Biggr) \\  & + \p \Biggl( \sup_{\tilde{X}_{k-1} \leq X_{l-1} < \tilde{X}_k} \sup_{X_{l-1} \leq x_i < X_l} \frac{\Bigl| \sum_{\substack{n > x_i \\ P(n) \leq y_0}} \frac{f(n)}{\sqrt{n}} \Bigr|}{\exp\Bigl((1 + \varepsilon){\sqrt{\log_2 {X}_{l-1} \log_4 {X}_{l-1}}}\Bigr)} > 1 \Biggr). 
\end{align*}
We note that $\p (G_{k}^c)$ (where $G_k$ is as defined in \eqref{equ:2.08}) is larger than this first term.
Since we are assuming that $\p (G_{k}^c)$ is summable, we need only show that the second term is summable. By the union bound and Markov's inequality with second moments (using Lemma \ref{L:Expectation} to evaluate the expectation, which is applicable by the dominated convergence theorem), we have
\begin{align*}
&\p \Biggl( \sup_{\tilde{X}_{k-1} \leq X_{l-1} < \tilde{X}_k} \sup_{X_{l-1} \leq x_i < X_l} \frac{\Big| \sum_{\substack{n > x_i \\ P(n) \leq y_0}} \frac{f(n)}{\sqrt{n}} \Big|}{\exp\Bigl({(1 + \varepsilon)\sqrt{\log_2 {X}_{l-1} \log_4 {X}_{l-1}}}\Bigr)} > 1  \Biggr) \label{equ:2.09} \tag{2.09} \\
&\leq \sum_{\tilde{X}_{k-1} \leq X_{l-1} < \tilde{X}_k} \sum_{X_{l-1} \leq x_i < X_l} \frac{ \sum_{\substack{n > x_i \\ P(n) \leq y_0}} \frac{1}{n}}{\exp\Bigl(2(1 + \varepsilon){\sqrt{\log_2 \tilde{X}_{k-1} \log_4 \tilde{X}_{k-1}}}\Bigr)}. 
\end{align*}
Here we apply Rankin's trick to note that
\begin{align*} \sum_{\substack{n > x_i \\ P(n) \leq y_0}} \frac{1}{n} &\leq x_i^{-1/\log y_0} \prod_{p \leq y_0} \Bigl( 1 - \frac{1}{p^{1 - 1/\log y_0}} \Bigr)^{-1} \ll \frac{\log y_0}{x_i^{1/\log y_0}}.
\end{align*}
Recalling that $y_0 = \exp\bigl(c e^l/6 l\bigr)$, we can bound the probability \eqref{equ:2.09} by
\begin{align*}
\ll &\frac{1}{\exp\Bigl(2\sqrt{\log \log \tilde{X}_{k-1}}\Bigr)}\sum_{\tilde{X}_{k-1} \leq X_{l-1} < \tilde{X}_k} \sum_{X_{l-1} \leq x_i < X_l} \frac{\log y_0}{x_i^{1 / \log y_0}} \\
\ll& \frac{1}{\exp\bigl(2\rho^{(k-1)/2}\bigr)}\sum_{\tilde{X}_{k-1} \leq X_{l-1} < \tilde{X}_k} \frac{1}{l e^{l(6/ce - 1/c - 1)}}
\ll \frac{1}{\exp\bigl(2\rho^{(k-1)/2}\bigr)},
\end{align*}
which is summable (with $c$ as in subsection \ref{s:testpointvariation}). Hence if $\p(G_{k}^c)$ is summable, then $\p(\mathcal{B}_{0,k})$ is summable, as required. \\

We now proceed to show that $\p(\mathcal{B}_{1,k})$ is summable, which will conclude the proof of Proposition \ref{p:sumthenprop}. Here we introduce the events in \eqref{equ:2.07}, giving
\begin{align*} \p(\mathcal{B}_{1,k}) &\leq \p(\mathcal{B}_{1,k} \cap G_k \cap I_{k}) + \p(G_k^c) + \p(I_{k}^c).
\end{align*}
Therefore, assuming the summability of the trailing terms, it suffices to show that $\p(\mathcal{B}_{1,k} \cap G_k \cap I_{k})$ is summable. As in \citet{LTW} (equation (3.16)), by the union bound, then taking $2q$'th moments and using Hölder's inequality, we have 
\begin{align*}
\p(\mathcal{B}_{1,k} \cap G_k \cap I_{k}) \leq \sum_{\tilde{X}_{k-1} \leq X_{l-1} < \tilde{X}_k} \sum_{X_{l-1} \leq x_i < X_l} \sum_{1 \leq j \leq J} \frac{\E(|S_{i,j}|^{2q}\mathbf{1}_{G_{j,l} \cap I_{j,l}}) J^{2q-1}}{\exp\bigl(2q{(1 + \varepsilon)\sqrt{\log_2 x_i \log_4 x_i}}\bigr) }. \label{equ:2.10} \tag{2.10}
\end{align*}
We will choose $q \in \mathbb{N}$ depending on $k$ at the very end of this subsection. We let $\mathcal{F}_{y_{j-1}} = \sigma(\{f(p): \, p \leq y_{j-1}\})$ be the $\sigma$-algebra generated by $f(p)$ for all $p \leq y_{j-1}$, forming a filtration. Note that $G_{j,l}$ and $I_{j,l}$ are $\mathcal{F}_{y_{j-1}}$-measurable. We introduce a function $V$ of $x_i$ that slowly goes to infinity with $i$, specified at the end of subsection \ref{s:Dijterm}.
Recalling the definition of $S_{i,j}$ from \eqref{equ:2.04}, by our expectation result (Lemma \ref{L:Expectation}), we have
\begin{align*}
\E \big[ |S_{i,j}|^{2q} \mathbf{1}_{G_{j,l} \cap I_{j,l}} \big] & = \E \bigl[ \E (|S_{i,j}|^{2q} \mathbf{1}_{G_{j,l} \cap I_{j,l}} | \mathcal{F}_{y_{j-1}}) \bigr] \\ & \leq \E \Biggl[ \mathbf{1}_{G_{j,l} \cap I_{j,l}} \Biggl( \sum_{\substack{y_{j-1} < m \leq x_i \\ p|m \Rightarrow p \in (y_{j-1},y_j]}} \frac{\tau_q (m)}{m} \Biggl| \sum_{\substack{n \leq x_i/m \\ P(n) \leq y_{j-1}}} \frac{f(n)}{\sqrt{n}} \Biggr|^2 \Biggr)^q \Biggr] \\ 
& = \E \Biggl[ \mathbf{1}_{G_{j,l} \cap I_{j,l}} \Biggl( \sum_{\substack{y_{j-1} < m \leq x_i \\ p|m \Rightarrow p \in (y_{j-1},y_j]}} \frac{V \tau_q (m)}{m^2}\int_m^{m(1+1/V)} \Biggl| \sum_{\substack{n \leq x_i/m \\ P(n) \leq y_{j-1}}} \frac{f(n)}{\sqrt{n}} \Biggr|^2\, dt \Biggr)^q \Biggr] \\
& \leq 2^{3q} \Bigl( \mathbb{E} \bigl( \mathcal{C}_{i,j}^q \bigr) + \mathbb{E} \bigl( \mathcal{D}_{i,j}^q \bigr) \Bigr) , \label{equ:2.11} \tag{2.11}
\end{align*}
where
\begin{align*}
\mathcal{C}_{i,j}&= \mathbf{1}_{G_{j,l} \cap I_{j,l}} \sum_{\substack{y_{j-1} < m \leq x_i \\ p|m \Rightarrow p \in (y_{j-1},y_j]}} \frac{V \tau_q (m)}{m^2}\int_m^{m(1+1/V)} \Biggl| \sum_{\substack{n \leq x_i/t \\ P(n) \leq y_{j-1}}} \frac{f(n)}{\sqrt{n}} \Biggr|^2\, dt , \label{equ:2.12} \tag{2.12} \\
\mathcal{D}_{i,j}&=\sum_{\substack{y_{j-1} < m \leq x_i \\ p|m \Rightarrow p \in (y_{j-1},y_j]}} \frac{V \tau_q (m)}{m^2} \int_m^{m(1+1/V)}  \Biggl| \sum_{\substack{x_i /t < n \leq x_i/m \\ P(n) \leq y_{j-1}}} \frac{f(n)}{\sqrt{n}} \Biggr|^2 \, dt ,
\end{align*}
and we have used the fact that $|A+B|^r \leq 2^r (|A|^r + |B|^r)$.
\subsection{Bounding the main term $\mathcal{C}_{i,j}$} \label{s:Cijterm}
We will see that our choices of $G_{j,l}$ and $I_{j,l}$ completely determine an upper bound for $\mathcal{C}_{i,j}$. We first swap the order of summation and integration to obtain
\begin{equation}
\mathcal{C}_{i,j} = \mathbf{1}_{G_{j,l} \cap I_{j,l}} \int_{y_{j-1}}^{x_i} \Biggl| \sum_{\substack{n \leq x_i/t \\ P(n) \leq y_{j-1}}} \frac{f(n)}{\sqrt{n}} \Biggr|^2 \sum_{\substack{t / (1 + 1/V) \leq m \leq t \\ p|m \Rightarrow p \in (y_{j-1}, y_j]}} \frac{V \tau_q (m)}{m^2} dt \, .
\label{equ:2.13} \tag{2.13}
\end{equation}
To estimate the sum over the divisor function we employ the following result from Harper, \cite{harper2020moments} (section 2.1, referred to also as Number Theory Result 1 there).
\begin{nt} \label{nt} Let $0 < \delta < 1$, let $r \geq 1$ and suppose $\max\{ 3, 2r \} \leq y \leq z \leq y^{2}$ and that $1 < u \leq v(1-y^{-\delta})$. Let $\Omega(m)$ equal the number of prime factors of $m$ counting multiplicity. Then \[ \sum_{\substack{u \leq m \leq v \\ p | m \Rightarrow y \leq p \leq z}} r^{\Omega(m)} \ll_\delta \frac{(v-u)r}{\log y} \prod_{y \leq p \leq z} \biggl( 1 - \frac{r}{p} \biggr)^{-1} . \]
\end{nt}
We note that $\tau_q (m) \leq q^{\Omega(m)}$ by submultiplicativity of $\tau_q$. The above result is applicable assuming that $V$ is, say, smaller than $\sqrt{y_0}$, and $q$ is an integer with $2q \leq y_0$ (indeed, $q$ will be approximately $l \leq y_0 / 2$ and $V$ will be roughly $(\log X_l)^{2l^2} \leq \sqrt{y_0}$), in which case we have
\begin{align*}
\sum_{\substack{t / (1 + 1/V) \leq m \leq t \\ p|m \Rightarrow p \in (y_{j-1}, y_j]}}\frac{V \tau_q (m)}{m^2} \ll& \frac{V}{t^2} \sum_{\substack{t / (1 + 1/V) \leq m \leq t \\ p|m \Rightarrow p \in (y_{j-1}, y_j]}} \tau_q (m)  \ll \frac{V}{t^2} \sum_{\substack{t / (1 + 1/V) \leq m \leq t \\ p|m \Rightarrow p \in (y_{j-1}, y_j]}} q^{\Omega(m)} \label{equ:2.14} \tag{2.14} \\  \ll& \frac{q}{t \log y_{j-1}} \prod_{y_{j-1} < p \leq y_j} \biggl( 1 - \frac{q}{p} \biggr)^{-1}. 
\end{align*}
Since $q$ will be very small compared to $y_0$ (in particular, $q = o(\log y_0)$), we have
\[ \prod_{y_{j-1} < p \leq y_{j}} \biggl( 1 - \frac{q}{p} \biggr)^{-1} \ll \biggl( \frac{\log y_{j}}{\log y_{j-1}} \biggr)^{q}.\]
Using the above and \eqref{equ:2.13}, we have
\begin{equation*}
\mathcal{C}_{i,j} \ll \frac{q \mathbf{1}_{G_{j,l} \cap I_{j,l}}}{\log y_{j-1}} \biggl( \frac{\log y_{j}}{\log y_{j-1}} \biggr)^{q} \int_{y_{j-1}}^{x_i} \Biggl| \sum_{\substack{n \leq x_i/t \\ P(n) \leq y_{j-1}}} \frac{f(n)}{\sqrt{n}} \Biggr|^2 \frac{dt}{t}.
\end{equation*}
Proceeding similarly to Harper \cite{HarperHighMom}, we perform the change of variables $z = x_i / t$, giving
\[ \mathcal{C}_{i,j} \ll \frac{q \mathbf{1}_{G_{j,l} \cap I_{j,l}}}{\log y_{j-1}} \biggl( \frac{\log y_{j}}{\log y_{j-1}} \biggr)^{q} \int_{1}^{x_i/y_{j-1}} \Big| \sum_{\substack{n \leq z \\ P(n) \leq y_{j-1}}} \frac{f(n)}{\sqrt{n}} \Big|^2 \frac{dz}{z}. \]
To apply Harmonic Analysis Result \ref{ha-result}, we need the power of $z$ in the denominator of the integrand to be greater than $1$, and so we introduce a factor of $(1 / z)^{2/\log x_i}$. By the definitions of $y_{j-1}$ and $y_j$ from \eqref{equ:2.01}, we have
\begin{align*}
\mathcal{C}_{i,j} &\ll \frac{q e^{\alpha q} \mathbf{1}_{G_{j,l} \cap I_{j,l}}}{\log y_{j-1}} \int_{1}^{x_i/y_{j-1}} \Biggl| \sum_{\substack{n \leq z \\ P(n) \leq y_{j-1}}} \frac{f(n)}{\sqrt{n}} \Biggr|^2 \frac{dz}{z^{1 + 2/\log x_i}} \label{equ:2.15} \tag{2.15} \\
&\ll \frac{q e^{\alpha q} \mathbf{1}_{G_{j,l} \cap I_{j,l}}}{\log y_{j-1}} \int_{1}^{\infty} \Biggl| \sum_{\substack{n \leq z \\ P(n) \leq y_{j-1}}} \frac{f(n)}{\sqrt{n}} \Biggr|^2 \frac{dz}{z^{1 + 2/\log X_{l}}} \, ,
\end{align*}
where we have completed the range of the integral to $[1, \infty)$, and used the fact that $x_i < X_l$, allowing us to remove dependence on $x_i$ without much loss, since $\log x_i$ varies by a constant factor for $x_i \in [X_{l-1},X_l)$. This is a key point: we have related $M_f (x_i)$ to an Euler product which depends only on the large interval $[X_{l-1}, X_l)$ in which $x_i$ lies. We now apply Harmonic Analysis Result \ref{ha-result}, giving
\begin{equation}
\mathcal{C}_{i,j} \ll \frac{q e^{\alpha q} \mathbf{1}_{G_{j,l} \cap I_{j,l}}}{\log y_{j-1}} \int_{-\infty}^{\infty} \Biggl| \frac{F_{y_{j-1}} ( 1/2 + 1/\log X_l + it )}{1 / \log X_l + it} \Biggr|^2 \,dt \, .
\label{equ:2.16} \tag{2.16}
\end{equation}
This integral is not completely straightforward to handle, as the variable of integration is tied up with the random Euler-product $F_{y_{j-1}}$. To proceed, we follow the ideas of \citet{gerspach2022low} in performing a dyadic decomposition of the integral, and introducing constant factors (with respect to $t$, but random) that allow us to extract the approximate size of the integral over certain ranges. The size of these terms is then handled using the conditioning on $I_{j,l}$ (recalling the definitions from \eqref{equ:2.07} and \eqref{equ:2.08}). \\

First of all, note that over the interval $[T,2T]$, the factor $p^{it} = e^{it \log p}$ varies a bounded amount for any $p \leq e^{1/T}$. Therefore, the Euler factors $(1 - f(p)/p^{1/2 + 1/ \log X_l + it})^{-1}$ are approximately constant on $[T,2T]$ for $p \leq e^{1/T}$. Subsequently, when appropriate, we will approximate the numerator by $|F_{e^{1/T}} (1/2)|^2$. We write 
\begin{align*}
\int_{-\infty}^{\infty} \Bigg| \frac{F_{y_{j-1}} ( 1/2 + 1/\log X_l + it )}{1 / \log X_l + it} \Bigg|^2 \,dt \leq \int_{-1/\log y_{j-1}}^{1/\log y_{j-1}} + \sum_{\substack{1/\log y_{j-1} \leq |T| \leq 1/2 \\ T \text{ dyadic}}} \int_{T}^{2T} + \int_{1/2}^\infty + \int_{-\infty}^{-1/2} , \label{equ:2.17} \tag{2.17} 
\end{align*}
where each integrand is the same as that on the left hand side. Here, ``$T$ dyadic'' means that we will consider $T = 2^{n}/\log y_{j-1}$ so that $T$ lies in the given range. Negative $T$ are considered similarly, and one should make the appropriate adjustments in accordance with Remark \ref{r:integralsignremark}. For the first integral on the right hand side of \eqref{equ:2.17}, we have
\begin{align*}
\int_{-1/\log y_{j-1}}^{1/\log y_{j-1}} \Bigg|& \frac{F_{y_{j-1}}(1/2 + 1/\log X_l + it)}{1/\log X_l + it} \Bigg|^2 \, dt \\ & \leq (\log X_l )^2 \int_{-1/\log y_{j-1}}^{1/\log y_{j-1}} \Bigg| \frac{F_{y_{j-1}}(1/2 + 1/\log X_l + it)}{F_{y_{j-1}}(1/2)} \Bigg|^2 \, dt \, |F_{y_{j-1}} (1/2)|^2 \\
&\leq  \frac{l^4 (\log X_l)^2}{\log y_{j-1}} \bigl|F_{y_{j-1}} (1/2)\bigr|^2 \, ,
\end{align*}
due to conditioning on $I_{j,l}^{(1)}$ in \eqref{equ:2.16}. We proceed similarly for the second term on the right hand side of \eqref{equ:2.17}, as we have 
\begin{align*}
\sum_{\substack{1/\log y_{j-1} \leq |T| \leq 1/2 \\ T \text{ dyadic}}} & \int_{T}^{2T} \Biggl| \frac{F_{y_{j-1}}(1/2 + 1/\log X_l + it)}{1/\log X_l + it} \Biggr|^2  \, dt \\
& \leq \sum_{\substack{1/\log y_{j-1} \leq |T| \leq 1/2 \\ T \text{ dyadic}}} \frac{1}{T^2} \int_{T}^{2T} \Biggl| \frac{F_{y_{j-1}}(1/2 + 1/\log X_l + it)}{F_{e^{1/T}} (1/2)} \Biggr|^2  \, dt \, \bigl|F_{e^{1/T}} (1/2)\bigr|^2 \\
& \leq l^4 \log y_{j-1} \sup_{1/\log y_{j-1} \leq T \leq 1/2} \bigl|F_{e^{1/T}}(1/2)\bigr|^2,
\end{align*}
by the conditioning on $I^{(2)}_{j,l}$. Finally, the last two integrals can be bounded directly from the conditioning on $I_{j,l}^{(3)}$. Therefore, we find that the integral on the left hand side of \eqref{equ:2.17} is 
\[ \ll \frac{l^4 \, (\log X_l)^2 }{\log y_{j-1}} \sup_{p \leq y_{j-1}}\bigl|F_{p} (1/2)\bigr|^2,\]
and so by \eqref{equ:2.16}, we have
\begin{equation*}
\mathcal{C}_{i,j} \ll \frac{ q \,  e^{\alpha q} \, l^4 \, (\log X_l)^2 \mathbf{1}_{G_{j,l} \cap I_{j,l}}}{(\log y_{j-1})^2} \sup_{p \leq y_{j-1}}\bigl|F_{p} (1/2)\bigr|^2 . 
\end{equation*}
We bound the Euler product term using our conditioning on $G_{j,l}$ from \eqref{equ:2.07},
\begin{equation}
\mathcal{C}_{i,j} \ll  \frac{ q \,  e^{\alpha q} \, (\log X_l)^2}{l^6( \log y_{j-1})^2} \exp\Bigl(2 (1 + \varepsilon) \sqrt{\log_2 {X}_{l-1} \log_4 {X}_{l-1}}\Bigr) .
\label{equ:2.18} \tag{2.18}
\end{equation}  
\subsection{Bounding the error term $\mathcal{D}_{i,j}$} \label{s:Dijterm}
We now proceed with bounding $\mathbb{E} \bigl( \mathcal{D}_{i,j}^q \bigr)$, where $\mathcal{D}_{i,j}$ is defined in \eqref{equ:2.12}. Similarly to Harper \cite{HarperHighMom} (in `Proof of Propositions 4.1 and 4.2') we first consider $\bigl( \mathbb{E} \bigl( \mathcal{D}_{i,j}^q \bigr) \bigr)^{1/q}$, giving us access to Minkowski's inequality. By definition, we have
\begin{align*}
\bigl( \mathbb{E} \bigl( \mathcal{D}_{i,j}^q \bigr) \bigr)^{1/q} = \Biggl[ \E \Biggl( \sum_{\substack{y_{j-1} < m \leq x_i \\ p|m \Rightarrow p \in (y_{j-1},y_j]}} \frac{V \tau_q (m)}{m^2} \int_m^{m(1+1/V)}  \Biggl| \sum_{\substack{x_i /t < n \leq x_i/m \\ P(n) \leq y_{j-1}}} \frac{f(n)}{\sqrt{n}} \Biggr|^2 \, dt \Biggr)^q \Biggr]^{1/q},
\end{align*}
and by Minkowski's inequality,
\begin{align*}
\bigl( \mathbb{E} \bigl( \mathcal{D}_{i,j}^q \bigr) \bigr)^{1/q} \leq \sum_{\substack{y_{j-1} < m \leq x_i \\ p|m \Rightarrow p \in (y_{j-1},y_j]}} \frac{\tau_q (m)}{m} \Biggr[ \E \Biggl( \frac{V}{m} \int_m^{m(1+1/V)}  \Biggl| \sum_{\substack{x_i /t < n \leq x_i/m \\ P(n) \leq y_{j-1}}} \frac{f(n)}{\sqrt{n}} \Biggr|^2 \, dt \Biggr)^q \Biggr]^{1/q}.
\end{align*}
Now applying Hölder's inequality (noting that the integral is normalised) and splitting the outer sum over $m$ at $x_i/V$, we have
\begin{align*}
\bigl( \mathbb{E} \bigl( \mathcal{D}_{i,j}^q \bigr) \bigr)^{1/q} & \leq \sum_{\substack{y_{j-1} < m \leq x_i/V \\ p|m \Rightarrow p \in (y_{j-1},y_j]}} \frac{\tau_q (m)}{m} \Biggr[ \frac{V}{m} \int_m^{m(1+1/V)} \E \Biggl| \sum_{\substack{x_i /t < n \leq x_i/m \\ P(n) \leq y_{j-1}}} \frac{f(n)}{\sqrt{n}} \Biggr|^{2q} \, dt \Biggr]^{1/q} \tag{2.19} \label{equ:2.19} \\ & + \sum_{\substack{x_i / V < m \leq x_i \\ p|m \Rightarrow p \in (y_{j-1},y_j]}} \frac{\tau_q (m)}{m} \Biggr[ \frac{V}{m} \int_m^{m(1+1/V)} \E \Biggl| \sum_{\substack{x_i /t < n \leq x_i/m \\ P(n) \leq y_{j-1}}} \frac{f(n)}{\sqrt{n}} \Biggr|^{2q} \, dt \Biggr]^{1/q}.
\end{align*}
We will show that these terms on the right hand side are small. Beginning with the second term, we note that the length of the innermost sum over $n$ is at most $\frac{x_i}{m} \bigl(1 - \frac{1}{1+1/V}\bigr)$, and since $m > x_i /V$, this is $\leq \frac{1}{1+1/V} < 1$. Therefore, the innermost sum contains at most one term, giving the upper bound
\begin{align*}
\sum_{\substack{x_i / V < m \leq x_i \\ p|m \Rightarrow p \in (y_{j-1},y_j]}} \frac{\tau_q (m)}{m} \Biggr[ \frac{V}{m} \int_m^{m(1+1/V)} \frac{t^q}{x_i^q} \, dt \Biggr]^{1/q} \leq \frac{2}{x_i}\sum_{\substack{x_i / V < m \leq x_i \\ p|m \Rightarrow p \in (y_{j-1},y_j]}} \tau_q (m),
\end{align*}
where we have taken the maximum value of $t$ in the integral and assumed that $1+1/V < 2$, since $V$ will go to infinity with $i$. Similarly to \eqref{equ:2.14}, we use sub-multiplicativity of $\tau_q (m)$ and apply Number Theory Result \ref{nt} (whose conditions are certainly satisfied on the same assumptions as for \eqref{equ:2.14}), giving a bound
\begin{align*}
\leq \frac{2}{x_i}\sum_{\substack{x_i / V < m \leq x_i \\ p|m \Rightarrow p \in (y_{j-1},y_j]}} q^{\Omega(m)} \ll \frac{q}{\log y_{j-1}} \prod_{y_{j-1} < p \leq y_{j-1}} \Bigl( 1 - \frac{q}{p} \Bigr)^{-1} \ll \frac{q e^{\alpha q}}{\log y_{j-1}}, \tag{2.20} \label{equ:2.20}
\end{align*}
which will turn out to be a sufficient bound for our purpose. We now bound the first term of \eqref{equ:2.19}, which requires a little more work. We first use Lemma \ref{L:Expectation} to evaluate the expectation in the integrand. This gives the upper bound
\begin{align*}
\sum_{\substack{y_{j-1} < m \leq x_i/V \\ p|m \Rightarrow p \in (y_{j-1},y_j]}} \frac{\tau_q (m)}{m} \Biggr[ \frac{V}{m} \int_m^{m(1+1/V)} \biggl( \sum_{\substack{x_i /t < n \leq x_i/m \\ P(n) \leq y_{j-1}}} \frac{\tau_q (n)}{n} \biggr)^q dt \Biggr]^{1/q}.
\end{align*}
Applying Cauchy--Schwarz, we get an upper bound of
\begin{align*}
& \sum_{\substack{y_{j-1} < m \leq x_i/V \\ p|m \Rightarrow p \in (y_{j-1},y_j]}} \frac{\tau_q (m)}{m} \Biggr[ \frac{V}{m} \int_m^{m(1+1/V)} \Biggl( \biggl( \sum_{\substack{x_i /t < n \leq x_i/m \\ P(n) \leq y_{j-1}}} \frac{1}{n^2} \biggr) \biggl( \sum_{\substack{x_i /t < n \leq x_i/m \\ P(n) \leq y_{j-1}}} {\tau_q}^2 (n) \biggr) \Biggr)^{q/2} dt  \Biggr]^{1/q} \\
\leq & \, \sum_{\substack{y_{j-1} < m \leq x_i/V \\ p|m \Rightarrow p \in (y_{j-1},y_j]}} \frac{\tau_q (m)}{m} \biggl( \sum_{x_i /m(1+1/V) < n \leq x_i/m} \frac{1}{n^2} \biggr)^{1/2} \biggl( \sum_{n \leq x_i/m
} {\tau_{q^2}} (n) \biggr)^{1/2},
\end{align*}
where we have taken $t$ maximal and used the fact that $\tau_q (n)^2 \leq \tau_{q^2} (n)$. By a length-max estimate, one can find that $\sum_{x_i /m(1+1/V) < n \leq x_i/m} \frac{1}{n^2} \ll \frac{m}{x_i V}$. Furthermore, using the fact that $\sum_{n \leq x} \tau_k (x) \leq x(2\log x)^{k-1} $ for $x\geq 3$, $k \geq 1$ (see Lemma 3.1 of \cite{BNR}), we obtain the bound
\begin{align*}
\ll \frac{1}{V^{1/2}} \sum_{\substack{y_{j-1} < m \leq x_i/V \\ p|m \Rightarrow p \in (y_{j-1},y_j]}} \frac{\tau_q (m)}{m} \bigl( 2 \log x_i \bigr)^{q^2/2} ,
\end{align*}
Completing the sum over $m$, we have the upper bound
\begin{align*}
\ll \frac{1}{V^{1/2}} \sum_{\substack{m \geq 1 \\ p|m \Rightarrow p \in (y_{j-1},y_j]}} \frac{\tau_q (m)}{m} \bigl( 2 \log x_i \bigr)^{q^2/2} & \ll \frac{1}{V^{1/2}}\bigl( 2 \log x_i \bigr)^{q^2/2} \prod_{y_{j-1} < p \leq y_j} \Bigl( 1 - \frac{1}{p} \Bigr)^{-q} \\
& \ll \frac{2^{q^2/2} e^{\alpha q }(\log x_i)^{q^2/2}}{V^{1/2}}.
\end{align*}
Combining this bound with the bound for the second term \eqref{equ:2.20}, we get a bound for the right hand side of \eqref{equ:2.19}, from which it follows that
\begin{align*}
\E \bigl( \mathcal{D}_{i,j}^q \bigr) \leq K^q \Biggl( \frac{q e^{\alpha q}}{\log y_{j-1}} + \frac{2^{q^2/2} e^{\alpha q }(\log x_i)^{q^2/2}}{V^{1/2}} \Biggr)^q ,
\end{align*}
for some absolute constant $K>0$. Taking $V = (\log x_i)^{2q^2}$, and $\alpha = 1/q$, this bound will certainly be negligible compared to the main term coming from \eqref{equ:2.18}.  We remark that this value of $V$ is appropriate for use in Number Theory Result \ref{nt} in \eqref{equ:2.14} and \eqref{equ:2.20}.

\subsection{Completing the proof of Proposition \ref{p:sumthenprop}}
Since the main term from \eqref{equ:2.18} dominates the error term above, from \eqref{equ:2.11} we obtain that
\[ \E (|S_{i,j}|^{2q} \mathbf{1}_{G_{j,l} \cap I_{j,l}} ) \leq \Biggl( \frac{R_\varepsilon \, q \, (\log X_l)^2 \,  \exp\bigl(2 (1 + \varepsilon) \sqrt{ \log_2 {X}_{l-1} \log_4 {X}_{l-1}}\bigr)}{l^6 (\log y_{j-1})^2} \Biggr)^q . \]
for some positive constant $R_\varepsilon$ from the ``Big Oh'' implied constant in \eqref{equ:2.18}. Now \eqref{equ:2.10} gives a bound on the probability
\begin{align*} \p(\mathcal{B}_{1,k} \cap G_k \cap I_{k}) &\leq \sum_{\tilde{X}_{k-1} \leq X_{l-1} < \tilde{X}_k} \sum_{X_{l-1} \leq x_i < X_l} \sum_{1 \leq j \leq J} J^{2q-1} \Biggl( \frac{R_\varepsilon \, q \, (\log X_l)^2 }{l^{6} (\log y_{j-1})^2} \Biggr)^q \\
& \leq \sum_{\tilde{X}_{k-1} \leq X_{l-1} < \tilde{X}_k} \sum_{X_{l-1} \leq x_i < X_l} \Biggl( \frac{16 R_\varepsilon \, J^2 \, q}{c^2 \, l^{4}} \Biggr)^q . \end{align*}
We take $q = \lfloor \rho^k \rfloor = \lfloor \log \log \tilde{X}_k \rfloor$, which satisfies the assumptions for Number Theory Result \ref{nt} in \eqref{equ:2.14} and \eqref{equ:2.20}. Using the fact $J \ll \rho^k \log l \ll k \rho^k$ from \eqref{equ:2.02}, and noting that there are no more than $e^{l/c}$ terms in the innermost sum, and no more than $\rho^k$ terms in the outermost sum, and that $\rho^{k-1} \leq l \leq \rho^{k+1}$ for large $k$, we find that taking trivial bounds gives
\begin{align*}
\p(\mathcal{B}_{1,k} \cap G_k \cap I_{k}) &\ll \biggl( \frac{R'_\varepsilon k^2}{\rho^{k}} \biggr)^{\lfloor\rho^k\rfloor},
\end{align*}
when $k$ is sufficiently large, for some constant $R'_\varepsilon>0$ depending only on $\varepsilon$ (since $\rho>1$ depends only on $\varepsilon$). Therefore, $\p(\mathcal{B}_{1,k} \cap G_k \cap I_{k})$ is summable. Recalling \eqref{equ:2.05}, this completes the proof of Proposition \ref{p:sumthenprop}.
\end{proof}

\subsection{Law of the iterated logarithm-type bound for the Euler product} \label{s:loiltyperesult}   

In this subsection, we prove that $\p(G_k^c)$ (as defined in \eqref{equ:2.08}) is summable. Recall $\tilde{X}_k = e^{e^{\rho^{k}}}$ for some $\rho > 1$ depending on $\varepsilon$, chosen shortly. It suffices to prove that 
\begin{equation*}
\p \Biggl( \sup_{\tilde{X}_{k-1} \leq X_{l-1} < \tilde{X}_k} \frac{ \sup_{p \leq X_l} |F_p (1/2)|}{\exp\Bigl((1 + \varepsilon/2)\sqrt{ \log_2 \tilde{X}_{k-1} \log_4 \tilde{X}_{k-1}}\Bigr)} > 1 \Biggr), \label{equ:2.21} \tag{2.21}
\end{equation*}
is summable in $k$, noting that $l^5 = (\log_2 X_l)^5 = o(\exp({\sqrt{\log_2 X_{l-1}}}))$, and so we removed the $l^5$ factor in \eqref{equ:2.08} by altering $\varepsilon$ in the denominator.
To prove \eqref{equ:2.21}, we will utilise two standard results from probability.
\begin{probres}[Lévy inequality, Theorem 3.7.1 of \citet{Gut_2013}] \label{levy}
Let $X_1, X_2,...$ be independent, symmetric random variables and $S_n = X_1 + X_2 + ... + X_n$. Then for any $x$,
\[ \p(\max_{1 \leq m \leq n} S_m > x) \leq 2 \p (S_n > x) .\]
\end{probres}
Our $S_m$ will more or less be the random walk $\sum_{p \leq m} \Re f(p)/\sqrt{p}$. This result tells us that the distribution of the maximum of a random walk is controlled by the distribution of the endpoint, allowing us to remove the supremum in \eqref{equ:2.21}. The next result will allow us to handle the resulting term.
\begin{probres}[Upper exponential bound, Lemma 8.2.1 of \citet{Gut_2013}] \label{UEB}
Let $X_1, X_2,...$ be mean zero independent random variables. Let $\sigma_k^2 = \var X_k$, and $s_n^2 = \sum_{k=1}^n \sigma_k^2$. Furthermore, suppose that, for $c_n > 0$, \[ |X_k| \leq c_n s_n \text{ a.s. \, for } k = 1,2,...,n .\]
Then, for $0 < x < 1/c_n$,
\[ \p \biggl( \sum_{k=1}^n X_k > x s_n \biggr) \leq \exp \biggl( - \frac{x^2}{2} \Bigl( 1 - \frac{x c_n}{2} \Bigr) \biggr) \, . \]
\end{probres}
We proceed by writing the probability in \eqref{equ:2.21} as 
\begin{equation*}
\p \Biggl( \sup_{x \leq Z} \frac{\Big|\prod_{p \leq x} \Bigl(1 - \frac{f(p)}{\sqrt{p}}\Bigr)^{-1}\Big|}{\exp\Bigl((1 + \varepsilon/2)\sqrt{ \log_2 \tilde{X}_{k-1} \log_4 \tilde{X}_{k-1}}\Bigr)} > 1  \Biggr),
\end{equation*}
where $Z = \exp(\exp(\lceil \rho^k \rceil))$ is the largest possible value that $X_l$ can take; it is minimal so that $Z = X_l > \tilde{X}_k$. Taking the exponential of the logarithm of the numerator, the above probability is equal to
\begin{align*}
&\,\p \biggl( \sup_{x \leq Z} -\sum_{p \leq x} \Re \log \biggl(1 - \frac{f(p)}{\sqrt{p}}\biggr) > (1 + \varepsilon/2) \sqrt{ \log_2 \tilde{X}_{k-1} \log_4 \tilde{X}_{k-1}} \biggr) \\
= &\, \p \biggl( \sup_{x \leq Z} \sum_{p \leq x} \sum_{k \geq 1} \frac{\Re f(p)^k}{k p ^{k/2}} > (1 + \varepsilon/2) \sqrt{ \log_2 \tilde{X}_{k-1} \log_4 \tilde{X}_{k-1}} \biggr) \\
\leq &\, \p \biggl( \sup_{x \leq Z} \sum_{p \leq x} \biggl( \frac{\Re f(p)}{\sqrt{p}} + \frac{\Re f(p)^2}{2p} \biggr) > (1 + \varepsilon/3) \sqrt{ \log_2 \tilde{X}_{k-1} \log_4 \tilde{X}_{k-1}} \biggr) \\
\leq &\, \p \biggl( \sup_{x \leq Z} \sum_{p \leq x} \frac{\Re f(p)}{\sqrt{p}} > (1 + \varepsilon/4) \sqrt{ \log_2 \tilde{X}_{k-1} \log_4 \tilde{X}_{k-1}} \biggr) \\
& + \, \p \biggl( \sup_{x \leq Z} \sum_{p \leq x}  \frac{\Re f(p)^2}{2p} > \frac{\varepsilon}{12} \sqrt{ \log_2 \tilde{X}_{k-1} \log_4 \tilde{X}_{k-1}} \biggr)
\end{align*}
These probabilities can be bounded by the Lévy inequality, Probability Result \ref{levy}. The second probability is then summable by Markov's inequality with second moments. It remains to show that
\begin{align*} &\p \biggl( \sum_{p \leq Z} \frac{\Re f(p)}{\sqrt{p}}  > (1 + \varepsilon/4)\sqrt{ \log_2 \tilde{X}_{k-1} \log_4 \tilde{X}_{k-1}} \biggr), \label{equ:2.22} \tag{2.22} \end{align*}
is summable, which we prove using the upper exponential bound (Probability Result \ref{UEB}). By a straightforward calculation using the fact that $2\Re (z) = z + \bar{z}$, we have $\var [\Re f(p)/\sqrt{p} ] = 1/2p$. Therefore we have $s_Z^2 = \sum_{p \leq Z} 1/2p $. Let $c_Z = 2 / s_Z$. Certainly such a choice satisfies $|\Re f(p)/\sqrt{p}| \leq c_Z s_Z  $ for all primes $p$, so Probability Result \ref{UEB} implies that for any $x \leq 1 / c_Z = s_Z / 2$,
\[ \p \Biggl( \sum_{p \leq Z} \frac{\Re f(p)}{\sqrt{p}} > x \Bigl( \sum_{p \leq Z} \frac{1}{2p} \Bigr)^{1/2} \Biggr) \leq \exp \Biggl( - \frac{x^2}{2} \Biggl( 1 - \frac{x}{\bigl( \sum_{p \leq Z} 1/2p\bigr)^{1/2}}  \Biggr) \Biggr). \]
We take \[ x = (1 + \varepsilon/4) \Biggl(\frac{\log_2 \tilde{X}_{k-1} \log_4 \tilde{X}_{k-1}}{\sum_{p \leq Z} 1/2p}\Biggr)^{1/2}. \]
Recall that $Z = \exp(\exp(\lceil \rho^k \rceil ))$. Using the fact that $Z > \tilde{X}_{k-1}$, it is not hard to show that, for large $k$, this value of $x$ is applicable, seeing as $x \ll \sqrt{\log_4 Z}$ and $s_Z \gg \sqrt{\log_2 Z}$, hence $x < s_Z/2$. This value of $x$ gives an upper bound for the probability in \eqref{equ:2.22} of 
\[ \leq 2 \exp \Biggl( - \frac{(1 + \varepsilon/4)^2 \log_2 \tilde{X}_{k-1} \log_4 \tilde{X}_{k-1}}{\sum_{p \leq Z} 1/p} \Biggl(  1 - \frac{(1 + \varepsilon/4)\sqrt{ \log_2 \tilde{X}_{k-1} \log_4 \tilde{X}_{k-1}}}{\sum_{p \leq Z} 1/2p } \Biggr) \Biggr), \]
Since for large $k$ we have $\sum_{p \leq Z} 1/2p \gg \log_2 Z  \gg \log_2 \tilde{X}_{k-1}$, we find that the term in the innermost parenthesis is of size $1 + o(1)$. Furthermore, since $\sum_{p \leq Z} 1/p  = \log_2 Z + O(1) $, the previous equation is bounded above by
\[ \ll \exp\biggl( - (1 + o(1)) \frac{(1 + \varepsilon/4)^2 \log_2 \tilde{X}_{k-1} \log_4 \tilde{X}_{k-1}}{\log_2 Z + O(1)} \biggr). \]
Inserting the definitions $\tilde{X}_{k-1} = \exp(\exp(\rho^{k-1}))$ and $Z = \exp(\exp(\lceil \rho^{k} \rceil))$, this is 
\[ \ll \exp\biggl( - (1 + o(1)) \frac{(1 + \varepsilon/4)^2 \rho^{k-1} \log ((k-1) \log \rho)}{\lceil \rho^{k} \rceil + O(1)} \biggr). \]
Note that for $\rho>1$ fixed, for sufficiently large $k$ we have $\lceil \rho^k \rceil \leq \rho^{k+1}$. Therefore, the last term can be bounded above by
\[\ll \frac{1}{((k-1)\log \rho)^{(1 + \varepsilon/4)^2(1 + o(1))/\rho^2}}. \]
Taking $\rho$ sufficiently close to $1$ (in terms of $\varepsilon$), this is summable in $k$. Subsequently, the probability \eqref{equ:2.21} is summable, as required.
\subsection{Probability of complements of integral events are summable} \label{s:cond}
Here we prove that $\mathbb{P}(I_k^c)$ is summable. Recalling \eqref{equ:2.07} and \eqref{equ:2.08}, we note that by the union bound, it suffices to show that the following are summable.
\begin{align*} I_{1,k}^c &\coloneqq \bigcup_{l \,:\, \tilde{X}_{k-1} \leq X_{l-1} < \tilde{X}_k} \bigcup_{j=1}^J \Biggl\{ \int_{-1/\log y_{j-1}}^{1/\log y_{j-1}}  \Bigg|\frac{F_{y_{j-1}}(1/2 + 1/\log X_l + it)}{F_{y_{j-1}} (1/2)} \Bigg|^2 \, dt > \frac{l^4}{\log y_{j-1}}\Biggr\} , \tag{2.23} \label{equ:2.23} \\
I_{2,k}^c &\coloneqq \bigcup_{l \,:\, \tilde{X}_{k-1} \leq X_{l-1} < \tilde{X}_k} \bigcup_{j=1}^J \Biggl\{ \sum_{\substack{1/\log y_{j-1} \leq |T| \leq 1/2 \\ T \text{ dyadic}}} \frac{1}{T^2} \int_{T}^{2T} \Bigg| \frac{F_{y_{j-1}}(1/2 + 1/\log X_l + it)}{F_{e^{1/T}} (1/2)} \Bigg|^2  \, dt  > l^4 \log y_{j-1} \Biggr\} , \\
I_{3,k}^c &\coloneqq \bigcup_{l \,:\, \tilde{X}_{k-1} \leq X_{l-1} < \tilde{X}_k} \bigcup_{j=1}^J \Biggl\{ \int_{1/2}^{\infty} \frac{\bigl|F_{y_{j-1}}\bigl(\frac{1}{2} + \frac{1}{\log X_l} + it\bigr)\bigr|^2 + \bigl|F_{y_{j-1}}\bigl(\frac{1}{2} + \frac{1}{\log X_l} - it\bigr)\bigr|^2}{t^2} \, dt \, > l^4 \log y_{j-1} \Biggr\}. \end{align*}
To prove that these events have summable probabilities, we wish to apply Markov's inequality, and so we need to be able to evaluate the expectation of the integrands. We employ the following result, which is similar to Lemma 3.1 of \citet{gerspach2022low}.
\begin{ep} \label{eulerproductresult}
For any $\sigma >0, \, t \in \R$, and any $x,y \geq 2$ such that $x \leq y$ and $\sigma \log y \leq 1$, we have 
\[ \E \Biggl| \frac{F_y (1/2 + \sigma + it)}{F_x (1/2)} \Biggr|^2 \ll \exp{(C t^2 (\log x)^2)} \Bigl( \frac{\log y}{\log x} \Bigr), \]
for some absolute constant $C>0$, and where the implied constant is also absolute.
\end{ep}
\begin{remark}
\normalfont Our choices for the range of the integrals and the denominators in our integrands, made in subsection \ref{s:Cijterm}, ensure that $|t|(\log x)$ is bounded when we apply the above result.
\end{remark}
\begin{proof}
The proof follows from standard techniques used in Euler Product Result 1 of \citet{harper2020sure}, the key difference being that we do not have $\sigma$ in the argument of the denominator. We therefore find that
\begin{align*}
\E\Biggl| \frac{F_y (1/2 + \sigma + it)}{F_x (1/2)} \Biggr|^2 &= \prod_{p \leq x} \Biggl( 1 + \frac{|p^{- \sigma - it} - 1|^2}{p} + O \Bigl( \frac{1}{p^{3/2}} \Bigr) \Biggr) \prod_{x < p \leq y} \Biggl( 1 + \frac{1}{p^{1 + 2 \sigma}} + O\Bigl( \frac{1}{p^{3/2}} \Bigr) \Biggr) \\
& \ll \exp \biggl( \sum_{p \leq x} \frac{|p^{-\sigma - it}-1|^2}{p} \biggr) \biggl( \frac{\log y}{\log x} \biggr). \tag{2.24} \label{equ:2.24}
\end{align*} 
To bound the first term, we use the fact that $\cos x \geq 1-x^2$ for all $x \in \R$, giving
\begin{align*}
|p^{-\sigma - it} -1 |^2 &= p^{-2\sigma} - 2p^{-\sigma} \cos (t \log p) + 1 \\
& \leq p^{-2\sigma} - 2p^{-\sigma} + 1 + 2p^{-\sigma} t^2 (\log p)^2 \\
& \leq (p^{-\sigma}-1)^2 + 2p^{-\sigma} t^2 (\log p)^2 \\
& \leq \sigma^2 (\log p)^2 + 2t^2 (\log p)^2,
\end{align*}
where on the last line we have used the fact that $|e^{-x} - 1| \leq x$ for $x > 0$. Inserting this into \eqref{equ:2.24} gives
\begin{align*}
\E\Biggl| \frac{F_y (1/2 + \sigma + it)}{F_x (1/2)} \Biggr|^2 &\ll \exp\biggl( \sum_{p \leq x} \frac{\sigma^2 (\log p)^2 + 2t^2 (\log p)^2}{p}\biggr) \biggl( \frac{\log y}{\log x} \biggr) \\
& \ll \exp \Bigl( C (\sigma^2  + 2t^2)(\log x)^2 \Bigr) \biggl( \frac{\log y}{\log x} \biggr),
\end{align*}
using the fact that $\sum_{p \leq x} (\log p)^2/p \leq C (\log x)^2$ for some $C>0$ to obtain the last line. The desired result (upon exchanging $2C$ for $C$) follows by noting that $\sigma \log x \leq \sigma \log y \leq 1$.
\end{proof}
Equipped with this result, we apply the union bound and Markov's inequality with first moments to show that each of the events in \eqref{equ:2.23} have probabilities that are summable. For the first event, this gives
\begin{align*}
\p(I_{1,k}^c) \leq \sum_{\tilde{X}_{k-1} \leq X_{l-1} < \tilde{X}_k} \sum_{j=1}^J \frac{\log y_{j-1}}{l^4} \int_{-1/\log y_{j-1}}^{1/ \log y_{j-1}} \E \Biggl|\frac{F_{y_{j-1}}(1/2 + 1/\log X_l + it)}{F_{y_{j-1}} (1/2)} \Biggr|^2 \, dt ,
\end{align*}
Now, by Euler Product Result \ref{eulerproductresult}, for some absolute constant $C>0$, we have
\begin{align*}
\p(I_{1,k}^c) &\leq \sum_{\tilde{X}_{k-1} \leq X_{l-1} < \tilde{X}_k} \sum_{j=1}^J \frac{\log y_{j-1}}{l^4}  \int_{-1/\log y_{j-1}}^{1/ \log y_{j-1}} \exp\bigl(C t^2 (\log y_{j-1})^2\bigr) \, dt \\ & \ll 
\sum_{\tilde{X}_{k-1} \leq X_{l-1} < \tilde{X}_k} \sum_{j=1}^J \frac{1}{l^4} \ll \sum_{\tilde{X}_{k-1} \leq X_{l-1} < \tilde{X}_k} \frac{\rho^k \log \rho^k}{l^4} \ll \frac{k}{\rho^{2k}},
\end{align*}
where in the second inequality we have used the fact that the integrand is bounded. Therefore $\p(I_{1,k}^c)$ is summable. The probability of the second event, $\p(I_{2,k}^c)$, can be handled almost identically. To show that $\p(I_{3,k}^c)$ is summable, we note that $\E |F_{y_{j-1}} (1/2 + 1 / \log X_l + it)|^2 \ll \log y_{j-1}$ (this is a fairly straightforward calculation and follows from Euler Product Result 1 of \citet{harper2020sure}), and one can then apply an identical strategy to the above. Note that we can apply Fubini's Theorem in this case, since the integrand is absolutely convergent. \\

Therefore we have verified the assumptions of Proposition \ref{p:sumthenprop}, completing the proof of the upper bound, Theorem \ref{T:1}.
\section{Lower Bound}
 In this section, we give a proof of Theorem \ref{T:2}. We shall prove that for any $\varepsilon > 0$,
 \begin{equation}
 \p \biggl( \max_{t \in [T_{k-1},T_k]} |M_f (t)|^2 \geq \exp\Bigl(2(1 - \varepsilon)\sqrt{ \log_2 T_k \log_4 T_k}\Bigr)\text{ i.o.} \biggr) = 1, \label{equ:3.01} \tag{3.01}
 \end{equation}
 for some intervals $(T_k)$, from which Theorem \ref{T:2} follows.
 \begin{proof}
Fix $\varepsilon>0$ and assume that it is sufficiently small throughout the argument, and that $k$ is sufficiently large. Implied constants from $\ll$ or ``Big Oh'' notation will depend on $\varepsilon$, unless stated otherwise. We take $T_k = \exp(\exp(\lambda^k))$, for some fixed $\lambda > 1$ (depending only on $\varepsilon$) chosen later. These intervals are of similar shape to the intervals $\tilde{X}_k$ in the upper bound, however here we will take $\lambda$ to be very large. Doing this allows for use of Borel--Cantelli lemma 2, seen as the terms we obtain, $ \sum_{p \leq T_k} \Re f(p)/\sqrt{p}$, will be controlled by the independent sums $\sum_{T_{k-1} < p \leq T_k} \Re f(p)/\sqrt{p}$. This is an approach taken in many standard proofs of the lower bound in the law of the iterated logarithm (see, for example, section 3.9 of Varadhan \cite{varadhan2001probability}). \\

Since $\int_{T_{k-1}}^{T_k} 1/t \, dt \leq \log T_k$, we have
\begin{equation} 
\max_{t \in [T_{k-1},T_k]} |M_f (t)|^2 \geq \frac{1}{\log T_k} \int_{T_{k-1}}^{T_k} \frac{|M_f(t)|^2}{t^{1 + 2 \log \log T_k / \log T_k}} \, dt, \label{equ:3.02} \tag{3.02}
\end{equation}
where the $2 \log \log T_k / \log T_k$ term has been introduced to allow use of Harmonic Analysis Result \ref{ha-result} at little cost, similarly to \eqref{equ:2.15}, whilst being sufficiently large so that we can complete the upper range of the integral without compromising our lower bound. \\

We now complete the range of the integral so that it runs from $1$ to infinity. For the lower range, by Theorem \ref{T:1}, we almost surely have, say,
 \begin{align*}
\frac{1}{\log T_k} \int_{1}^{T_{k-1}} \frac{|M_f(t)|^2}{t^{1 + 2 \log \log T_k / \log T_k}} \, dt \ll \exp{\bigl(3 \sqrt{\log_2 T_{k-1} \log_4 T_{k-1}}} \bigr). \label{equ:3.03} \tag{3.03}
 \end{align*}
Whereas for the upper integral, we almost surely have
\begin{equation}
\int_{T_k}^\infty \frac{\Bigl|\sum_{\substack{n \leq t \\ n \, T_{k}-\text{smooth}}}\frac{f(n)}{\sqrt{n}} \Bigr|^2}{t^{1 + 2 \log \log T_k / \log T_k}} dt \, \leq 1 ,
\label{equ:3.04} \tag{3.04}
\end{equation}
for sufficiently large $k$. This follows from the first Borel--Cantelli lemma, since Markov's inequality followed by Fubini's Theorem gives
\begin{align*}
\p \Biggl( \int_{T_k}^\infty \frac{\bigl|\sum_{\substack{n \leq t \\ n \, T_{k}-\text{smooth}}}\frac{f(n)}{\sqrt{n}} \bigr|^2}{t^{1 + 2 \log \log T_k / \log T_k}} \, dt \, > 1 \Biggr)&\leq \int_{T_k}^\infty \frac{\E\bigl|\sum_{\substack{n \leq t \\ n \, T_{k}-\text{smooth}}}\frac{f(n)}{\sqrt{n}} \bigr|^2}{t^{1 + 2 \log \log T_k / \log T_k}} \, dt \\
& \ll \int_{T_k}^\infty \frac{\log T_k}{t^{1 + 2 \log \log T_k / \log T_k}} \, dt\ll \frac{1}{\log \log T_{k}} = \frac{1}{\lambda^k},
\end{align*}
which is summable. Now combining \eqref{equ:3.02}, \eqref{equ:3.03} and \eqref{equ:3.04} we have that almost surely, for large $k$,
\begin{equation}
\max_{t \in [T_{k-1},T_k]} |M_f (t)|^2 \geq \frac{1}{\log T_k} \int_{1}^{\infty} \frac{\Bigl|\sum_{\substack{n \leq t \\ n \, T_k - \text{smooth}}} \frac{f(n)}{\sqrt{n}}\Bigr|^2}{t^{1 + 2 \log \log T_k / \log T_k}} \, dt - C\exp{\Bigl(3 \sqrt{\log_2 T_{k-1} \log_4 T_{k-1}}\Bigr)}, \label{equ:3.05} \tag{3.05} \end{equation}
for some constant $C>0$. We proceed by trying to lower bound the first term on the right hand side of this equation. By Harmonic Analysis Result \ref{ha-result}, we have
\begin{align*}
\frac{1}{\log T_k} \int_{1}^{\infty} \frac{\Bigl|\sum_{\substack{n \leq t \\ n \, T_k - \text{smooth}}} \frac{f(n)}{\sqrt{n}}\Bigr|^2}{t^{1 + 2 \log \log T_k / \log T_k}} \, dt &= \frac{1}{2 \pi \log T_k} \int_{-\infty}^\infty \Biggl| \frac{F_{T_k}(1/2 + \log \log T_k / \log T_k + it)}{\log \log T_k / \log T_k + i t} \Biggr|^2 \, dt \\
& \geq \frac{(1 + o(1))\log T_k}{2 \pi (\log \log T_k)^2} \int_{\frac{-1}{2\log T_k}}^{\frac{1}{2\log T_k}} \biggl| F_{T_k}\biggl(\frac{1}{2} + \frac{\log \log T_k}{ \log T_k} + it\biggr) \biggr|^2 \, dt .
\end{align*}
This last term on the right hand side is equal to
\[ \frac{1+o(1)}{2 \pi (\log \log T_k)^2} \int_{\frac{-1}{2\log T_k}}^{\frac{1}{2\log T_k}} \exp\biggl(2 \log \biggl| F_{T_k}\biggl(\frac{1}{2} + \frac{\log \log T_k}{ \log T_k} + it\biggr) \biggr|\biggr) \, \log T_k \,  dt. \]
Note that $\log T_k \, dt$ is a probability measure on the interval that we are integrating over. Since the exponential function is convex, we can apply Jensen's inequality as in the work of \citet{harper2020moments}, section 6, (see also \citet{AHZ}, section 4) to obtain the following lower bound for the first term on the right hand side of \eqref{equ:3.05}
\begin{align*} 
& \, \frac{1+o(1)}{2 \pi (\log \log T_k)^2} \exp \Biggl( \int_{\frac{-1}{2\log T_k}}^{\frac{1}{2\log T_k}} 2 \log \biggl| F_{T_k}\biggl(\frac{1}{2} + \frac{\log \log T_k}{ \log T_k} + it\biggr) \biggr| \log T_k \,  dt \Biggr)  \\
= & \, \frac{1+o(1)}{2 \pi (\log \log T_k)^2} \exp \Biggl( \int_{\frac{-1}{2\log T_k}}^{\frac{1}{2\log T_k}} -2  \sum_{p \leq T_k} \Re \log \biggl( 1 - \frac{f(p)}{p^{1/2 + \log \log T_k / \log T_k + it}} \biggr) \log T_k \,  dt \Biggr) \\
= & \, \frac{1+o(1)}{2 \pi (\log \log T_k)^2} \exp \Biggl( 2 \log T_k  \sum_{p \leq T_k} \int_{\frac{-1}{2\log T_k}}^{\frac{1}{2\log T_k}} \frac{\Re f(p)}{p^{1/2 + \sigma_k + it}}  + \frac{\Re f(p)^2}{2p^{1 + 2 \sigma_k + 2it}} + O \biggl( \frac{1}{p^{3/2}} \biggr) dt \Biggr),
\end{align*}
where $\sigma_k = \log \log T_k / \log T_k$. Since $1/p^{3/2}$ is summable over primes, this term can be bounded below by
\begin{equation*} \frac{c'}{(\log \log T_k)^2} \exp \Biggl( 2 \log T_k  \sum_{p \leq T_k} \int_{\frac{-1}{2\log T_k}}^{\frac{1}{2\log T_k}} \frac{\Re  f(p) }{p^{1/2 + \sigma_k + it}}  + \frac{\Re f(p)^2}{2p^{1 + 2 \sigma_k + 2it}}  dt \Biggr), \end{equation*}
for some constant $c' > 0$. The argument of the exponential is very similar to $\sum_{p \leq T_k} \Re f(p)/p^{1/2}$, which puts us in good stead for the law of the iterated logarithm. \\~\\
Note that
\begin{align*}
\int_{\frac{-1}{2\log T_k}}^{\frac{1}{2\log T_k}} p^{-it} \, dt & = \frac{2 \sin \bigl( \frac{\log p}{2\log T_k}  \bigr)}{\log p}, \text{ and } \, \,
\int_{\frac{-1}{2\log T_k}}^{\frac{1}{2\log T_k}} p^{-2it} \, dt = \frac{1}{\log T_k} + O \biggl( \frac{(\log p)^2}{(\log T_k)^3} \biggr).
\end{align*}
Therefore, we get a lower bound for the first term on the right hand side of \eqref{equ:3.05} of
\begin{align*}
\frac{c'}{(\log \log T_k)^2} \exp \Biggl( 2 \log T_k  \sum_{p \leq T_k} \biggl(& \frac{2\Re f(p) \sin\bigl(\frac{\log p}{2 \log T_k}\bigr)}{p^{1/2 + \sigma_k} \log p }  + \frac{\Re  f(p)^2 }{2p^{1 + 2 \sigma_k} \log T_k} + O \biggl( \frac{(\log p)^2}{p( \log T_k)^3}\biggr) \biggr) \Biggr) \\
\geq  \frac{c''}{(\log \log T_k)^2} \exp \Biggl(& 2  \sum_{p \leq T_k} \biggl( \frac{2\Re f(p) (\log T_k) \sin\bigl(\frac{\log p}{2 \log T_k}\bigr)}{p^{1/2 + \sigma_k} \log p }  + \frac{\Re f(p)^2}{2p^{1 + 2 \sigma_k}}\biggr) \Biggr) , \label{equ:3.06} \tag{3.06}
\end{align*}
for some constant $c'' > 0$, where we have used the fact that $\sum_{p \leq T_k} (\log p)^2/p \ll (\log T_k)^2$. 

To prove \eqref{equ:3.01}, it suffices to prove that 
\begin{equation}
\p \Bigl(\sum_{p \leq T_k}\frac{2\Re f(p) (\log T_k) \sin\bigl(\frac{\log p}{2 \log T_k}\bigr)}{p^{1/2 + \sigma_k} \log p }  + \frac{\Re f(p)^2 }{2p^{1 + 2 \sigma_k}} \geq {(1-\varepsilon/3)\sqrt{\log_2 T_k \log_4 T_k}} \text{ i.o.} \Bigr) = 1, \label{equ:3.07} \tag{3.07}
\end{equation}
since, if this were true, it would follow from \eqref{equ:3.05} and \eqref{equ:3.06} that almost surely,
\begin{align*}
\frac{\max_{t \in [T_{k-1},T_k]} |M_f (t)|^2}{\exp\bigl(2(1 - \varepsilon)\sqrt{ \log_2 T_k \log_4 T_k}\bigr)} \geq \frac{c''\exp\bigl(4\varepsilon/3 \sqrt{\log_2 T_k \log_4 T_k}\bigr)}{2 (\log \log T_k)^2}+ o ( 1 )
\end{align*}
infinitely often, and for any $\lambda > 1$, the right hand side is larger than $1$ for large $k$. \\

Therefore, to complete the proof, we just need to show that \eqref{equ:3.07} holds. This follows from a fairly straightforward application of the Berry-Esseen Theorem and the second Borel--Cantelli lemma, as in the proof of the law of the iterated logarithm in section 3.9 of Varadhan \cite{varadhan2001probability}. We first analyse the independent sums over $p$ in the disjoint ranges $(T_{k-1}, T_k]$, which will control the sum in \eqref{equ:3.07} when $\lambda$ is large. \\~\\
\begin{probres}[Berry-Esseen Theorem, Theorem 7.6.2 of \citet{Gut_2013}] \label{T:BE}
Let $X_1, X_2,...$ be independent random variables with zero mean and let $S_n = X_1 + ... + X_n$. Suppose that $\gamma_k^3 = \E |X_k|^3 < \infty$ for all $k$, and set $\sigma_k^2 = \var [ X_k] $, $s_n^2 = \sum_{k=1}^n \sigma_k^2$, and $\beta_n^3 = \sum_{k=1}^n \gamma_k^3$. Then
\[ \sup_{x \in \R} \Bigl| \p(S_n > x s_n) - \frac{1}{\sqrt{2\pi}}\int_{x}^\infty e^{-t^2/2} \, dt \,\Bigr| \leq C \frac{\beta_n^3}{s_n^3},\]
for some absolute constant $C>0$.
\end{probres}
If we take 
\begin{equation}
x = (1 - \varepsilon/2) \Biggl(\frac{\log_2 T_k \log_4 T_k}{\sum_{T_{k-1} < p \leq T_k} \frac{1}{2 p^{1 + 2 \sigma_k}} \bigl( \frac{2 \log T_k}{\log p} \bigr)^2 \sin^2 \bigl( \frac{\log p}{2 \log T_k} \bigr) + \frac{1}{8p^{2 + 4 \sigma_k}}} \Biggr)^{1/2} \, , \label{equ:3.08} \tag{3.08} \end{equation}
then, since the denominator in the parenthesis is the variance of our sum, for some constant $\tilde{C}>0$ independent of $k$, we have
\begin{align*} 
\p \Biggl(\sum_{T_{k-1} < p \leq T_k} & \frac{2\Re f(p) (\log T_k) \sin\bigl(\frac{\log p}{2 \log T_k}\bigr)}{p^{1/2 + \sigma_k} \log p }  + \frac{\Re f(p)^2}{2p^{1 + 2 \sigma_k}} \geq (1-\varepsilon/2)\sqrt{\log_2 T_k \log_4 T_k} \Biggr)\\ &\geq
\frac{1}{\sqrt{2\pi}} \int_x^\infty e^{-t^2/2} \, dt \, - \frac{\tilde{C}}{\Bigl(\sum_{T_{k-1} < p \leq T_k} \frac{1}{2 p^{1 + 2 \sigma_k}} \bigl( \frac{2 \log T_k}{\log p} \bigr)^2 \sin^2 \bigl( \frac{\log p}{2 \log T_k} \bigr) + \frac{1}{8p^{2 + 4 \sigma_k}}\Bigr)^{3/2}} \, . \label{equ:3.09} \tag{3.09}
\end{align*}
Here we have used the fact that the sums over third moments of our summand are uniformly bounded regardless of $k$, giving a bound of size $\tilde{C}$ for the $\beta_n$ terms in the Theorem. \\

To prove \eqref{equ:3.07}, it is sufficient to show that the right hand side of \eqref{equ:3.09} is not summable in $k$. The result will then follow by the second Borel--Cantelli lemma, and a short argument used to complete the lower range of the sum. Note that the second Borel--Cantelli lemma is applicable since our events are independent for distinct values of $k$. To proceed, it will be helpful to lower bound the sums of the variances,
\[ \sum_{T_{k-1} < p \leq T_k} \Bigl( \frac{1}{2 p^{1 + 2 \sigma_k}} \Bigl( \frac{2 \log T_k}{\log p} \Bigr)^2 \sin^2 \Bigl( \frac{\log p}{2 \log T_k} \Bigr) + \frac{1}{8p^{2 + 4 \sigma_k}} \Bigr). \]
By shortening the sum and noting that $\frac{1}{u^2} \sin^2 u \geq 1 - \varepsilon/4$ for $u$ sufficiently small, when $k$ is large we have the lower bound
\begin{align*}
\sum_{T_{k-1} < p \leq (1 - \varepsilon/4)^{-1 / 2 \sigma_k}} \frac{1}{2 p^{1 + 2 \sigma_k}} \Bigl( \frac{2 \log T_k}{\log p}& \Bigr)^2 \sin^2 \Bigl( \frac{\log p}{2 \log T_k} \Bigr) \geq \bigl( 1 - \varepsilon/4 \bigr) \sum_{T_{k-1} < p \leq (1 - \varepsilon/4)^{-1 / 2 \sigma_k}} \frac{1}{2 p^{1 + 2 \sigma_k}} \\ 
& \geq \bigl( 1 - \varepsilon/2 \bigr) \sum_{T_{k-1} < p \leq  (1 - \varepsilon/4)^{-1 / 2 \sigma_k}} \frac{1}{2p} \\
& \geq \frac{1 - \varepsilon/2}{2} \log \log T_k + O \bigl( \log \log T_{k-1} \bigr),
\end{align*}
recalling that $\sigma_k = \log \log T_k / \log T_k$. Since $\log \log T_k = \lambda^k$, this lower bound implies that the second term on the right hand side of \eqref{equ:3.09} is summable. Therefore, we just need to show that the first term on the right hand side is not. By standard estimates, we have $\frac{1}{\sqrt{2\pi}}\int_u^\infty e^{-t^2/2} \,du \gg \frac{1}{u} e^{-u^2/2}$ for all $u \geq 1$. Since the above lower bound gives an upper bound for $x$ from \eqref{equ:3.08}, we find that
\begin{align*}
\frac{1}{\sqrt{2\pi}} \int_x^\infty e^{-t^2/2} \, dt &\gg \frac{1}{\log_4 T_k} \exp\biggl( -\frac{(1 - \varepsilon/2)^2 \log_2 T_k \log_4 T_k}{2\sum_{T_{k-1} < p \leq T_k} \frac{1}{2 p^{1 + 2 \sigma_k}} \bigl( \frac{2 \log T_k}{\log p} \bigr)^2 \sin^2 \bigl( \frac{\log p}{2 \log T_k} \bigr) + \frac{1}{8p^{2 + 4 \sigma_k}}} \biggr) \\
& \gg \frac{1}{\log (k \log \lambda)} \exp\biggl( -\frac{(1 - \varepsilon/2) \log_2 T_k \log_4 T_k}{\log_2 T_k + O( \log_2 T_{k-1})}\biggr) \\
& \gg \frac{1}{\log (k \log \lambda)} \exp \biggl( - \frac{(1 - \varepsilon/2)\log (k \log \lambda)}{1 + O( 1/\lambda)} \biggr),
\end{align*}
where all implied constants depend at most on $\varepsilon$. Here we have used the fact that $T_k = \exp(\exp(\lambda^k))$. Taking $\lambda$ sufficiently large in terms of $\varepsilon$, we have
\[ \frac{1}{\sqrt{2\pi}} \int_x^\infty e^{-t^2/2} \, dt \gg \frac{1}{k^{1 - \varepsilon/4}}, \]
which is not summable over $k$. This proves that we almost surely have \[ \sum_{T_{k-1} < p \leq T_k} \frac{2\Re f(p) (\log T_k) \sin\bigl(\frac{\log p}{2 \log T_k}\bigr)}{p^{1/2 + \sigma_k} \log p }  + \frac{\Re f(p)^2}{2p^{1 + 2 \sigma_k}} \geq {(1-\varepsilon/2)\sqrt{\log_2 T_k \log_4 T_k}}, \]
infinitely often. The statement \eqref{equ:3.07} then follows by noting that we can complete the above sum to the whole range $p \leq T_k$, seen as one can apply Probability Result \ref{UEB} very similarly to subsection \ref{s:loiltyperesult} to show that almost surely, for large $k$,
\[ \sum_{p \leq T_{k-1}} \frac{2\Re f(p) (\log T_k) \sin\bigl(\frac{\log p}{2 \log T_k}\bigr)}{p^{1/2 + \sigma_k} \log p }  + \frac{\Re f(p)^2}{2p^{1 + 2 \sigma_k}} \leq \varepsilon/6 \sqrt{\log_2 T_k \log_4 T_{k}}, \]
when $\lambda$ is sufficiently large in terms of $\varepsilon$. This allows us to deduce that almost surely,
\begin{align*}
\sum_{p \leq T_k} \frac{2\Re f(p) (\log T_k) \sin\bigl(\frac{\log p}{2 \log T_k}\bigr)}{p^{1/2 + \sigma_k} \log p }  + \frac{\Re f(p)^2 }{2p^{1 + 2 \sigma_k}} \geq ( 1 - \varepsilon/3) \sqrt{\log_2 T_k \log_4 T_k},
\end{align*}
infinitely often, if $\lambda$ is taken to be sufficiently large in terms of $\varepsilon$. Therefore, \eqref{equ:3.07} holds, completing the proof of Theorem \ref{T:2}.
\end{proof} 
\vspace{0.3cm}
\subsubsection*{Acknowledgements} The author would like to thank his supervisor, Adam Harper, for the suggestion of this problem, for many useful discussions, and for carefully reading an earlier version of this paper.
\printbibliography
\end{document}